\documentclass[12pt,reqno]{article}

\usepackage[usenames]{color}
\usepackage{amssymb}
\usepackage{amsmath}
\usepackage{amsthm}
\usepackage{amsfonts}
\usepackage{amscd}
\usepackage{graphicx}

\usepackage[colorlinks=true,
linkcolor=webgreen,
filecolor=webbrown,
citecolor=webgreen]{hyperref}

\definecolor{webgreen}{rgb}{0,.5,0}
\definecolor{webbrown}{rgb}{.6,0,0}

\usepackage{color}
\usepackage{fullpage}
\usepackage{float}

\usepackage{graphics}
\usepackage{latexsym}

\begin{document}

\theoremstyle{plain}
\newtheorem{theorem}{Theorem}
\newtheorem{corollary}[theorem]{Corollary}
\newtheorem{lemma}{Lemma}
\newtheorem{example}{Example}
\newtheorem*{remark}{Remark}

\begin{center}
\vskip 1cm
{\LARGE\bf Combinatorial identities with multiple harmonic-like numbers \\ }

\vskip 1cm

{\large
Kunle Adegoke \\
Department of Physics and Engineering Physics \\ Obafemi Awolowo University, 220005 Ile-Ife, Nigeria \\
\href{mailto:adegoke00@gmail.com}{\tt adegoke00@gmail.com}

\vskip 0.2 in

Robert Frontczak \\
Independent Researcher \\ Reutlingen,  Germany \\
\href{mailto:robert.frontczak@web.de}{\tt robert.frontczak@web.de}
}
\end{center}

\vskip .2 in

\begin{abstract}
Multiple harmonic-like numbers are studied using the generating function approach. 
A closed form is stated for binomial sums involving these numbers and two additional parameters. 
Several corollaries and examples are presented which are immediate consequences of the main result. 
Finally, combinatorial identities involving harmonic-like numbers and other prominent sequences like 
hyperharmonic numbers and odd harmonic numbers are offered.
\end{abstract}



\bigskip

\section{Preliminaries}

Cheon and El-Mikkawy \cite{Cheon1,Cheon2} defined multiple harmonic-like numbers by
\begin{equation}
H_n(m) = \sum_{1\leq k_1+k_2+\cdots +k_m\leq n} \frac{1}{k_1 k_2\cdots k_m}
\end{equation}
with $H_n(0)=1$ for $n\geq 0$ and $H_0(m)=0$ for $m\geq 1$. They showed that the generating function of $H_n(m)$ equals
\begin{equation}\label{h_like_gf}
H(z) = \sum_{n=0}^\infty H_n(m) z^n = \frac{(-\ln(1-z))^m}{1-z}.
\end{equation}
For $m=1$ these numbers reduce to ordinary harmonic numbers $H_n$ as
\begin{equation}
H_n(1) = \sum_{1\leq k_1\leq n} \frac{1}{k_1} = H_n, \qquad H_0 = 0. 
\end{equation}
For $m=2$ we see that
\begin{align*}
H_n(2) &= \sum_{1\leq k_1+k_2\leq n} \frac{1}{k_1 k_2} = \sum_{j=1}^n \sum_{k_1+k_2=j} \frac{1}{k_1 k_2} \\
&= \sum_{j=1}^n \sum_{k_1=1}^{j-1} \frac{1}{k_1 (j-k_1)} = \sum_{j=1}^n \frac{2}{j} H_{j-1} \\
&= \sum_{j=1}^n \frac{H_{n-j}}{j} = H_n^2 - H_n^{(2)},
\end{align*}
where in the last line a result of Kargin and Can \cite{Kargin} was used and where $H_n^{(2)}$ are the second-order harmonic numbers, i.e.,
\begin{equation*}
H_n^{(2)} = \sum_{j=1}^n \frac{1}{j^2}.
\end{equation*}
The brute force computation of $H_n(3)$ is tedious. The result is
\begin{equation*}
H_n(3) = \sum_{j=1}^n \frac{1}{j} \sum_{l=1}^{n-j} \frac{H_{n-j-l}}{l}.
\end{equation*}

The next lemma is helpful.
\begin{lemma}\label{lem.z55fecm}
For all $n\geq 1$ and $m\geq 0$ we have the identity
\begin{equation}\label{eq.rhzasjc}
H_n(m+1) = \sum_{j=1}^n \frac{H_{n-j}(m)}{j}.
\end{equation}
\end{lemma}
\begin{proof}
Using \eqref{h_like_gf} we have
\begin{align*}
\sum_{n=0}^\infty H_n(m+1) z^n &= (1-z)\frac{-\ln (1-z)}{1-z} \frac{(-\ln(1-z))^m}{1-z} \\
&= (1-z) \left (\sum_{n=0}^\infty H_n z^n \right) \left( \sum_{n=0}^\infty H_n(m) z^n \right ) \\
&= \sum_{n=0}^\infty \sum_{j=0}^n H_j H_{n-j}(m) z^n - \sum_{n=1}^\infty \sum_{j=0}^{n-1} H_j H_{n-1-j}(m) z^n. 
\end{align*}
Extracting and comparing the coefficients of $z^n$ we obtain for all $n\geq 1$
\begin{align*}
H_n(m+1) &= \sum_{j=0}^{n-1} H_j \left ( H_{n-j}(m) - H_{n-1-j}(m) \right ) \\
&= \sum_{j=1}^n H_{n-j}(m) (H_j - H_{j-1}) \\
&= \sum_{j=1}^n \frac{H_{n-j}(m)}{j}
\end{align*}
as claimed.
\end{proof}

Multiple harmonic-like numbers were studied recently by Chen and Guo in the papers \cite{Chen1} and \cite{Chen2}.
For instance, in \cite{Chen1} several summation formulae involving harmonic-like numbers and other combinatorial numbers were derived.
In the paper \cite{Chen2}, a certain sequence $A_\alpha(n,k)$ was studied and, as a part of this study, additional interesting combinatorial identities involving harmonic-like numbers were presented.

In this paper, we continue the work on harmonic-like numbers applying the generating function approach. 
Our first main result is a closed form for binomial sums involving these numbers and two additional parameters $a,b\in\mathbb{C}$. 
Several corollaries and examples are presented which are immediate consequences of the main result. 
Finally, combinatorial identities involving harmonic-like numbers and other prominent sequences like 
hyperharmonic numbers and odd harmonic numbers are offered. \\

In what follows we will need the definition of the Stirling numbers of the first kind, $s(n,k)$. 
These numbers is defined by the generating function
\begin{equation}\label{def_Stirling}
\sum_{n=k}^\infty s(n,k) \frac{z^n}{n!} = \frac{\ln^k (1+z)}{k!}.
\end{equation}
Some particular values are 
\begin{align*}
s(n,k) &= 0 \qquad\mbox{for}\,\, n<k, \\
s(n,0) &= \begin{cases}
 1, & n=0;  \\
 0, & n\geq 1;  \\
 \end{cases} \\
s(n,1) &= (-1)^{n-1} (n-1)! \\
s(n,2) &= (-1)^n (n-1)! H_{n-1}.
\end{align*}

\section{Binomial sums involving $H_n(m)$}

For $a,b\in\mathbb{C}$ let $S_n(a,b,m)$ be defined by
\begin{equation}
S_n(a,b,m) = \sum_{k=0}^n \binom{n}{k} a^k b^{n-k} H_k(m).
\end{equation}
Then we have the following result.

\begin{theorem}\label{thm1}
For all $n\geq 0$ we have
\begin{equation}\label{main_id1}
S_n(a,b,m) = \sum_{j=0}^m \binom{m}{j} \sum_{k=0}^n H_k(j) (a+b)^k \frac{(m-j)!}{(n-k)!} (-1)^{n-k} b^{n-k} s(n-k,m-j),
\end{equation}
where $s(n,k)$ are the Stirling numbers of the first kind.
\end{theorem}
\begin{proof}
Let $S(z)$ denote the generating function of $S_n(a,b,m)$. Then (see \cite{Boyadzhiev2,Prodinger})
\begin{align*}
S(z) &= \sum_{n=0}^\infty S_n(a,b,m) z^m \\
&= \frac{1}{1-bz} H\left (\frac{az}{1-bz}\right ) \\
&= \frac{1}{1-(a+b)z}\left ( - \ln\left (\frac{1-(a+b)z}{1-bz}\right )\right )^m \\
&= (-1)^m \frac{(\ln(1-(a+b)z) - \ln(1-bz))^m}{1-(a+b)z} \\
&= \sum_{j=0}^m \binom{m}{j} \frac{(-\ln(1-(a+b)z))^j}{1-(a+b)z} \ln^{m-j}(1-bz).
\end{align*}
The result follows from \eqref{h_like_gf} and \eqref{def_Stirling} in combination with the Cauchy product. 
\end{proof}

\begin{remark}
When $m=0$ then using $H_n(0)=1,n\geq 0$, 
\begin{equation*}
S_n(a,b,0) = \sum_{k=0}^n \binom{n}{k} a^k b^{n-k} = (a+b)^n.
\end{equation*}
When $m=1$ then we get
\begin{align}
S_n(a,b,1) &= \sum_{k=0}^n \binom{n}{k} a^k b^{n-k} H_k \nonumber \\
&= \sum_{k=0}^n \frac{(-1)^{n-k}}{(n-k)!} (a+b)^k b^{n-k} s(n-k,1) + \sum_{k=0}^n \frac{(-1)^{n-k}}{(n-k)!} H_k (a+b)^k b^{n-k} s(n-k,0) \nonumber \\
&= H_n (a+b)^n - \sum_{k=0}^{n-1} (a+b)^k b^{n-k} \frac{1}{n-k},
\end{align}
which reproduces Boyadzhiev's main result (Proposition 6) from \cite{Boyadzhiev1}.
\end{remark}

\begin{corollary}
For $n,m\geq 0$ we have
\begin{equation}\label{cor_id1}
\sum_{k=0}^n \binom{n}{k} (-1)^k H_k(m) = (-1)^n \frac{m!}{n!} s(n,m)
\end{equation}
and
\begin{equation}\label{cor_id2}
\sum_{k=m}^n \binom{n}{k} \frac{s(k,m)}{k!} = \frac{1}{m!} H_n(m).
\end{equation}
\end{corollary}
\begin{proof}
The first result follows immediately by setting $a=-1$ and $b=1$ in \eqref{main_id1} and simplifying.
The second identity is the inverse binomial transform of the first (for information on the binomial transform consult \cite{Boyadzhiev2}). 
\end{proof}
\begin{remark}
The identity \eqref{cor_id1} was discovered first by Chen and Guo (see \cite[Corollary 8]{Chen2}). 
The second identity also appears in that paper (with a typo, see the proof of Proposition 13).
\end{remark}

\begin{corollary}
For $n,m\geq 0$ we have
\begin{equation}\label{cor_id3}
\sum_{k=0}^n \binom{n}{k} H_k(m) = \sum_{j=0}^m \binom{m}{j} \sum_{k=0}^n H_k(j) (-1)^{n-k} 2^k \frac{(m-j)!}{(n-k)!} s(n-k,m-j).
\end{equation}
In particular, we recover the classical identity \cite{Boyadzhiev1}
\begin{equation}
\sum_{k=0}^n \binom{n}{k} H_k = 2^n \left ( H_n - \sum_{k=1}^n \frac{1}{2^k k} \right ).
\end{equation}
\end{corollary}
\begin{proof}
Set $a=b$ in \eqref{main_id1} and simplify.
\end{proof}

\begin{corollary}\label{main_cor}
For $n\geq 0$ we have
\begin{equation}\label{cor_id4}
S_n(a,b,2) = H_n(2) (a+b)^n + 2 \sum_{k=1}^{n} (a+b)^{n-k} b^{k} \frac{H_{k-1}-H_{n-k}}{k}. 
\end{equation}
\end{corollary}
\begin{proof}
Set $m=2$ in \eqref{main_id1} and simplify.
\end{proof}

Some consequences of Corollary \ref{main_cor} will be stated as examples.
\begin{example}
We have
\begin{equation}
\sum_{k=0}^n \binom{n}{k} H_k(2) = 2^n \left ( H_n(2) + 2 \sum_{k=1}^n \frac{H_{k-1}-H_{n-k}}{2^{k} k} \right ).
\end{equation}
\end{example}

\begin{example}
From
\begin{equation}
\sum_{k=0}^n \binom{n}{k} (-1)^k H_k(2) = \frac{2}{n} H_{n-1}
\end{equation}
we deduce that
\begin{equation}
\sum_{k=0}^n \binom{n}{k} (-1)^k H_k^{(2)} = - \frac{H_n}{n},
\end{equation}
where we used that (see for instance \cite{Adegoke})
\begin{equation*}
\sum_{k=0}^n \binom{n}{k} (-1)^k H_k^{2} = \frac{H_n}{n} - \frac{2}{n^2}.
\end{equation*}
The inverse binomial relation also yields
\begin{equation}
\sum_{k=1}^n \binom{n}{k} (-1)^{k+1} \frac{H_k}{k} = H_n^{(2)}.
\end{equation}
\end{example}

\begin{example}
We have
\begin{equation}
\sum_{k=0}^n \binom{n}{k} 2^k (-1)^{n-k} H_k(2) = H_n(2) + 2 \sum_{k=1}^n (-1)^k \frac{H_{k-1}-H_{n-k}}{k}.
\end{equation}
\end{example}

\begin{example}
We have
\begin{equation}
\sum_{k=0}^n \binom{n}{k} 2^k H_k(2) = 3^n \left ( H_n(2) + 2 \sum_{k=1}^n \frac{H_{k-1}-H_{n-k}}{3^k k} \right ).
\end{equation}
\end{example}

\begin{example}
Let $F_n$ and $L_n$ be the Fibonacci and Lucas numbers, respectively (see Koshy \cite{Koshy} or Vajda \cite{Vajda}). Then we have
\begin{equation}
\sum_{k=0}^n \binom{n}{k} F_k H_k(2) = H_n(2) F_{2n} + 2 \sum_{k=1}^n F_{2(n-k)} \frac{H_{k-1}-H_{n-k}}{k},
\end{equation}
\begin{equation}
\sum_{k=0}^n \binom{n}{k} L_k H_k(2) = H_n(2) L_{2n} + 2 \sum_{k=1}^n L_{2(n-k)} \frac{H_{k-1}-H_{n-k}}{k},
\end{equation}
and
\begin{equation}
\sum_{k=0}^n \binom{n}{k} (-1)^{k+1} F_k H_k(2) = H_n(2) F_{n} + 2 \sum_{k=1}^n F_{n-k} \frac{H_{k-1}-H_{n-k}}{k},
\end{equation}
\begin{equation}
\sum_{k=0}^n \binom{n}{k} (-1)^{k} L_k H_k(2) = H_n(2) L_{n} + 2 \sum_{k=1}^n L_{n-k} \frac{H_{k-1}-H_{n-k}}{k}.
\end{equation}
\end{example}

\begin{corollary}\label{main_cor2}
For $n\geq 0$ we have
\begin{align}\label{cor_id5}
S_n(a,b,3) &= H_n(3) (a+b)^n \nonumber \\
&\quad -3 \sum_{k=1}^{n} (a+b)^{n-k} b^{k} \frac{H_{k-1}^2 - H_{k-1}^{(2)} - 2H_{k-1} H_{n-k} + H_{n-k}^2 - H_{n-k}^{(2)}}{k}. 
\end{align}
\end{corollary}
\begin{proof}
Set $m=3$ in \eqref{main_id1} and simplify using
\begin{equation*}
s(n,3) = \frac{1}{2} (-1)^{n-1} (n-1)! \left ( H_{n-1}^2 - H_{n-1}^{(2)} \right ).
\end{equation*}
\end{proof}

\section{Combinatorial identities from clever telescoping}

\begin{lemma}\label{lem.jjbwp3m}
If $(a_n)_{n\geq 1}$ is a sequence, then
\begin{equation*}
\sum_{k = 1}^n H_k \left( {a_{k + 1} - a_k} \right) = H_n a_{n + 1} - \sum_{k = 1}^n \frac{a_k}{k}.
\end{equation*}
\end{lemma}
\begin{proof}
This is an immediate consequence of the recurrence relation of the harmonic numbers.
\end{proof}

Lemma~\ref{lem.jjbwp3m} is broadly applicable and allows to instantly state a myriad of interesting results.
We give three examples as a warmup. Setting $a_n = H_{n-1}$ gives the known result
\begin{equation*}
\sum_{k = 1}^n \frac{{H_{k - 1} }}{k} = \frac{1}{2}\left( {H_n^2  - H_n^{(2)} } \right).
\end{equation*}
The choice $a_n = n$ yields the well-known harmonic number summation formula
\begin{equation*}
\sum_{k = 1}^n H_k = (n + 1)H_n - n.
\end{equation*}
Let $a_n = F_{n + 1}$ be the sequence of (shifted) Fibonacci numbers. Then we get the following Fibonacci-harmonic number identity:
\begin{equation*}
\sum_{k = 1}^n H_k F_k = H_n F_{n + 2} - \sum_{k = 1}^n \frac{{F_{k + 1} }}{k}.
\end{equation*}

\begin{theorem}
If $m$ and $n$ are non-negative integers, then
\begin{equation}\label{eq.o107dby}
\sum_{k = 1}^n H_k \sum_{j = m}^k {\binom{{k - 1}}{{j - 1}}\frac{{s(j,m)}}{{j!}}} = \frac{1}{{m!}}H_n (m)H_n  
- \frac{1}{{m!}}\sum_{k = 1}^n \frac{{H_{k - 1}(m)}}{k}.
\end{equation}
\end{theorem}
\begin{proof}
Use Lemma~\ref{lem.jjbwp3m} with $a_k=H_{k - 1}\left(m\right)$ while noting from~\eqref{cor_id2} that
\begin{equation}\label{eq.e04wpeg}
H_k(m) - H_{k - 1}(m) = m!\sum_{j = m}^k \binom{{k - 1}}{{j - 1}}\frac{s(j,m)}{j!}.
\end{equation}
\end{proof}

Setting $m=1$ in~\eqref{eq.o107dby} gives the known result
\begin{equation*}
\sum_{k = 1}^n \frac{{H_k}}{k} = \frac{1}{2}\left( H_n^2  + H_n^{(2)} \right),
\end{equation*}
while $m=2$ gives
\begin{equation*}
2\sum_{k = 1}^n H_k \sum_{j = 1}^k {( - 1)^j \binom{{k - 1}}{{j - 1}}\frac{{H_{j - 1} }}{j}} 
= H_n^3 - H_n^{(2)} H_n - \sum_{k = 1}^n \frac{{H_{k - 1}^2 - H_{k - 1}^{(2)} }}{k}.
\end{equation*}

\begin{theorem}
If $m$ and $n$ are positive integers, then
\begin{equation}
\sum_{k = 1}^n H_k \sum_{j = m}^{n - k + 1} \binom{{n - k}}{{j - 1}} \frac{s(j,m)}{j!} = \frac{1}{m!} H_{n + 1}(m+1).
\end{equation}
In particular,
\begin{equation}
\sum_{k = 1}^n \frac{{H_k }}{{n - k + 1}} = H_{n + 1}^2 - H_{n + 1}^{(2)}
\end{equation}
and
\begin{equation}
\sum_{k = 1}^n H_k \sum_{j = 2}^{n - k + 1} \binom{{n - k}}{{j - 1}}\frac{{\left( { - 1} \right)^j }}{j}H_{j - 1} 
= \frac{1}{2} \sum_{k = 1}^{n + 1} \frac{1}{k}\sum_{j = 1}^{n + 1 - k} \frac{{H_{n - k - j + 1} }}{j}.
\end{equation}
\end{theorem}
\begin{proof}
Set $a_k  = H_{n - k + 1} \left( m \right)$ in Lemma~\ref{lem.jjbwp3m} and use~\eqref{eq.rhzasjc} and~\eqref{eq.e04wpeg}.
\end{proof}

\begin{lemma}\label{lem.yugnf7k}
If $(a_n)_{n\geq 0}$ is a sequence, then
\begin{equation}
\sum_{k = 1}^n \frac{a_k - a_{k - 1}}{k} = \sum_{k = 1}^n \frac{{a_k}}{{k(k + 1)}} - a_0 + \frac{a_n}{n + 1}.
\end{equation}
\end{lemma}
\begin{proof}
This is a consequence of
\begin{equation*}
\frac{1}{k} -\frac{1}{k + 1} = \frac{1}{k(k + 1)}.
\end{equation*}
\end{proof}

Again, many results can be derived from Lemma~\ref{lem.yugnf7k}. For instance, setting $a_k=H_k+p$ ($p$ a non-negative integer) gives
\begin{equation}\label{Har_example}
\sum_{k = 1}^n \frac{{H_{k+p}}}{k(k + 1)} = 
\begin{cases}
  H_n^{(2)} - \frac{H_n}{n + 1},& p=0;  \\ 
	\\
  \frac{H_n + H_p - H_{n+p}}{p} + H_p - \frac{H_{n+p}}{n+1}, & p\geq 1;
 \end{cases}
\end{equation}
where we used the result
\begin{equation*}
\sum_{k = 1}^n \frac{1}{k(k + p)} = 
\begin{cases}
  H_n^{(2)}, & p=0;  \\ 
	\\
  \frac{H_n + H_p - H_{n+p}}{p}, & p\geq 1.
 \end{cases}
\end{equation*}
The result \eqref{Har_example} is most likely not new.

\begin{theorem}
If $m$ is a non-negative integer and $n$ is a positive integer, then
\begin{equation}
\sum_{k = 1}^n {\frac{{H_{n - k} \left( m \right)}}{{k\left( {k + 1} \right)}}}  = H_n \left( m \right) + H_n \left( {m + 1} \right) - H_{n + 1} \left( {m + 1} \right).
\end{equation}
In particular,
\begin{equation}
\sum_{k = 1}^n {\frac{{H_{n - k} }}{{k\left( {k + 1} \right)}}}  = H_n  + H_n^2  - H_n^{(2)}  - H_{n + 1}^2  + H_{n + 1}^{(2)}
\end{equation}
and
\begin{equation}
\sum_{k = 1}^n {\frac{{H_{n - k}^2  - H_{n - k}^{(2)} }}{{k(k + 1)}}}  = H_n^2  - H_n^{(2)}  + \sum_{k = 1}^n {\frac{1}{k}\sum_{j = 1}^{n - k} {\frac{{H_{n - k - j} }}{j}} }  - \sum_{k = 1}^{n + 1} {\frac{1}{k}\sum_{j = 1}^{n + 1 - k} {\frac{{H_{n + 1 - k - j} }}{j}} }.
\end{equation}
\end{theorem}
\begin{proof}
Use $a_k  = H_{n - k} (m)$ in Lemma~\ref{lem.yugnf7k} and invoke Lemma~\ref{lem.z55fecm}.
\end{proof}

\begin{lemma}\label{lem.u8veeoy}
Let $(a_n)_{n\geq 0}$ be a sequence. If $r$ is a complex number and $n$ is a non-negative number, then
\begin{equation}
\sum_{k = 0}^n (- 1)^k \binom{{r - 1}}{k} (a_{k + 1} - a_k) = (- 1)^n \binom{{r - 1}}{n}a_{n + 1} - \sum_{k = 0}^n (- 1)^k \binom{r}{k} a_k.
\end{equation}
\end{lemma}
\begin{proof}
This is a variation on an identity of Koll\'ar~\cite[Lemma 1]{rkollar}.
\end{proof}

\begin{theorem}
If $r$ is a complex number and $m$ and $n$ are positive integers, then
\begin{equation}
\begin{split}
\sum_{k = 0}^n {( - 1)^k \binom{{r - 1}}{k}\sum_{j = m}^{k +1} {\binom{{k}}{{j - 1}}\frac{{s(j,m)}}{{j!}}} }  &= ( - 1)^n \binom{{r - 1}}{n}\frac{1}{{m!}}H_{n + 1} \left( m \right)\\
&\qquad - \frac{1}{{m!}}\sum_{k = 0}^n {( - 1)^k \binom{{r}}{k}H_k \left( m \right)} .
\end{split}
\end{equation}
In particular,
\begin{equation}
\sum_{k = 0}^n {\frac{{\left( { - 1} \right)^k }}{{k + 1}}\binom{{r - 1}}{k}} 
= \left( { - 1} \right)^n \binom{{r - 1}}{n}H_{n + 1} - \sum_{k = 0}^n {\left( { - 1} \right)^k \binom{{r}}{k}H_k } 
\end{equation}
and
\begin{equation}
\begin{split}
\sum_{k = 0}^n {\left( { - 1} \right)^k \binom{{r - 1}}{k}\sum_{j = 2}^{k + 1} {\left( { - 1} \right)^j \binom{{k}}{{j - 1}}\frac{{H_{j - 1} }}{j}} } &= \left( { - 1} \right)^n \binom{{r - 1}}{n}\frac{1}{2}\left( {H_{n + 1}^2  - H_n^{(2)} } \right)\\
&\qquad - \frac{1}{2}\sum_{k = 0}^n {\left( { - 1} \right)^k \binom{{r}}{k}\left( {H_k^2  - H_k^{(2)} } \right)} .
\end{split}
\end{equation}
\end{theorem}
\begin{proof}
Use $a_k=H_k\left(m\right)$ in Lemma~\ref{lem.u8veeoy}.
\end{proof}

\section{More identities involving $H_n(m)$ and other sequences}

Hyperharmonic numbers $H_{n,p}$ (or $h_n^{(p)}$), $p\geq 1,$ are another generalization of harmonic numbers (\cite{Boyadzhiev3,Coppo,Dil}). 
They are defined by
\begin{equation}\label{eq.uxf1e3m}
H_{n,p} = \sum_{i=1}^n H_{i,p-1} \qquad \mbox{with} \qquad  H_{n,0} = \frac{1}{n},\, H_{0,p} = 0,\,H_{n,1} = H_n.
\end{equation}
They can be written in the compact form
\begin{equation}\label{eq.sw1cyzy}
H_{n,p+1} = \binom {n+p}{n} (H_{n+p} -H_{p} ), \qquad p=0,1,2,\ldots, 
\end{equation}

\begin{theorem}\label{hyphar}
For all $n,p,m\geq 0$ we have
\begin{equation}\label{eq.dk6gvdp}
\sum_{k=0}^n \binom{k+p}{k} H_{n - k}(m) (H_{k+p}-H_p) = \sum_{k=0}^n \binom{k+p}{k} H_{n-k}(m+1). 
\end{equation}
In particular,
\begin{equation}
\sum_{k=0}^n \binom{k+p}{k} (H_{k+p}-H_p) = \sum_{k=0}^n \binom{k+p}{k} H_{n-k}
\end{equation}
and
\begin{equation}
\sum_{k=0}^n \binom{k+p}{k} H_{n - k} (H_{k+p}-H_p) = \sum_{k=0}^n \binom{k+p}{k} \left ( H_{n-k}^2 - H_{n-k}^{(2)} \right ). 
\end{equation}
\end{theorem}
\begin{proof}
Let $B(z)$ be the generating function of $H_{n,p+1}$. We can calculate
\begin{align*}
H(z)\cdot B(z) &= \left (\sum_{n=0}^\infty H_n(m) z^n \right )\left (\sum_{n=0}^\infty H_{n,p+1} z^n \right ) \\
&= \sum_{n=0}^\infty \sum_{k=0}^n H_k(m) H_{n-k,p+1} z^n \\
&= \frac{(-\ln (1-z))^m}{1-z}\frac{-\ln (1-z)}{(1-z)^{p+1}} \\
&= \frac{1}{(1-z)^{p+1}} \frac{(-\ln (1-z))^{m+1}}{1-z} \\
&= \left (\sum_{n=0}^\infty \binom{n+p}{n} z^n \right )\left (\sum_{n=0}^\infty H_n(m+1) z^n \right ) \\
&= \sum_{n=0}^\infty \sum_{k=0}^n \binom{k+p}{k} H_{n-k}(m+1) z^n
\end{align*}
and the statement follows.
\end{proof}

The next theorem contains an identity involving $H_n(m)$ and odd harmonic numbers $O_n$, the latter being defined by
\begin{equation*}
O_n = \sum_{k = 1}^n \frac{1}{2k-1},\quad O_0=0.
\end{equation*}
Obvious relations between harmonic numbers $H_n$ and odd harmonic numbers $O_n$ are given by
\begin{equation*}
H_{2n} = \frac{1}{2} H_n + O_n \qquad \text{and} \qquad H_{2n - 1} = \frac{1}{2}H_{n - 1} + O_n.
\end{equation*}
Additional relations are contained in the next lemma. 

\begin{lemma}\label{lem.czxfdu7}
If $n$ is an integer, then
\begin{gather}
H_{n - 1/2} - H_{ - 1/2} = 2O_n,\label{eq.plh634k} \\
H_{n - 1/2} - H_{1/2} = 2\left( {O_n  - 1} \right),\label{eq.hgplrbd}\\
H_{n + 1/2} - H_{ - 1/2} = 2O_{n + 1},\label{eq.ivi1ex5} \\
H_{n + 1/2} - H_{1/2} = 2\left( {O_{n + 1}  - 1} \right),\label{eq.u6ng5d6}\\
H_{n + 1/2} - H_{n - 1/2} = \frac{2}{{2n + 1}},\\
H_{n - 1/2} - H_{-3/2} = 2\left( {O_n  - 1} \right)\label{eq.pobmr6h},\\
H_{n + 1/2} - H_{-3/2} = 2\left( {O_{n + 1}  - 1} \right).
\end{gather}
\end{lemma}

\begin{lemma}
If $p$ and $r$ are non-negative integers, then
\begin{equation}\label{eq.m2jjbl5}
H_{r,p + 1/2} = \frac{1}{{2^{2r - 1} }}\binom{{2p}}{p}^{- 1} \binom{{2(r + p)}}{{r + p}}\binom{{r + p}}{r}\left( {O_{r + p} - O_p } \right).
\end{equation}
\end{lemma}
\begin{proof}
Identity~\eqref{eq.hgplrbd} gives
\begin{equation}\label{eq.tb96ki2}
H_{r + p - 1/2} - H_{r - 1/2} = 2\left( {O_{r + p} - O_p } \right).
\end{equation}
Using this and 
\begin{equation}\label{eq.j80dqx2}
\binom{{r + p - 1/2}}{r} = \frac{1}{{2^{2r} }}\binom{{2p}}{p}^{ - 1} \binom{{2(r + p)}}{{r + p}}\binom{{r + p}}{r}
\end{equation}
in~\eqref{eq.sw1cyzy} gives~\eqref{eq.m2jjbl5}.
\end{proof}

\begin{theorem}
If $n$ and $p$ are non-negative integers, then
\begin{equation}\label{eq.suzj3to}
\begin{split}
&\sum_{k = 1}^n {\frac{1}{{2^{2k} }}\binom{{2\left( {k + p} \right)}}{{k + p}}\binom{{k + p}}{k}\left( {O_{k + p}  - O_p } \right)} \\
&= \frac{1}{{2^{2n + 1} }}\,\frac{{p + 1}}{{2p + 1}}\binom{{2\left( {n + p + 1} \right)}}{{n + p + 1}}\binom{{n + p + 1}}{n}\left( {O_{n + p + 1}  - O_{p + 1} } \right).
\end{split}
\end{equation}
\end{theorem}
In particular,
\begin{equation}\label{eq.oklok93}
\sum_{k = 1}^n {\frac{{O_k }}{{2^{2k} }}\binom{{2k}}{k}}  = \frac{{n + 1}}{{2^{2n + 1} }}\binom{{2(n + 1)}}{{n + 1}}\left( {O_{n + 1}  - 1} \right).
\end{equation}

\begin{proof}
Write $p + 1 + 1/2$ for $p$ in~\eqref{eq.uxf1e3m} to obtain
\begin{equation*}
\sum_{k = 1}^n {H_{k,p + 1/2} }  = H_{n,(p + 1) + 1/2},
\end{equation*}
and use~\eqref{eq.m2jjbl5}.
\end{proof}

\begin{theorem}\label{odd_id1}
For all $n,m\geq 0$ we have
\begin{equation}\label{eq.unrz3hp}
\sum_{k=0}^n \binom{2k}{k} \frac{O_k H_{n-k}(m)}{4^k} = \frac{1}{2} \sum_{k=0}^n \binom{2k}{k} \frac{H_{n-k}(m+1)}{4^k}. 
\end{equation}
In particular,
\begin{equation}\label{eq.zfc0q8z}
\sum_{k=0}^n \binom{2k}{k} \frac{O_k}{4^k} = \frac{1}{2} \sum_{k=0}^n \binom{2k}{k} \frac{H_{n-k}}{4^k}
\end{equation}
and
\begin{equation}
\sum_{k=0}^n \binom{2k}{k} \frac{O_k H_{n-k}}{4^k} = \frac{1}{2} \sum_{k=0}^n \binom{2k}{k} \frac{1}{4^k} \left ( H_{n-k}^2 - H_{n-k}^{(2)} \right ). 
\end{equation}
\end{theorem}
\begin{proof}
We use the fact that \cite{HChen,Gould}
\begin{equation*}
\sum_{n=0}^\infty \binom{2n}{n} z^n = \frac{1}{\sqrt{1-4z}}.
\end{equation*}
It is also known that \cite{HChen}
\begin{equation*}
O(z) = \sum_{n=0}^\infty \binom{2n}{n} O_n z^n = \frac{1}{2} \sqrt{1-4z} \left (\frac{-\ln (1-4z)}{1-4z} \right ).
\end{equation*}
This yields
\begin{align*}
O(z)\cdot H(4z) &= \sum_{n=0}^\infty \sum_{k=0}^n \binom{2k}{k} O_k 4^{n-k} H_{n-k}(m) z^n \\
&= \frac{1}{2 \sqrt{1-4z}}\, \frac{(-\ln (1-4z))^{m+1}}{1-4z} \\
&= \frac{1}{2} \left (\sum_{n=0}^\infty \binom{2n}{n} z^n \right ) \left ( \sum_{n=0}^\infty H_n(m+1) 4^n z^n \right ) \\
&= \frac{1}{2} \sum_{n=0}^\infty \sum_{k=0}^n \binom{2k}{k} 4^{n-k} H_{n-k}(m+1) z^n 
\end{align*}
and the proof is completed.
\end{proof}
Note that from~\eqref{eq.oklok93} and~\eqref{eq.zfc0q8z}, we also have
\begin{equation}\label{eq.tb6ik5l}
\sum_{k=0}^n \binom{2k}{k} \frac{H_{n-k}}{2^{2k}}=\frac{{n + 1}}{{2^{2n} }}\binom{{2(n + 1)}}{{n + 1}}\left( {O_{n + 1}  - 1} \right).
\end{equation}

The next theorem, based on Theorem~\ref{hyphar}, generalizes Theorem~\ref{odd_id1}.
\begin{theorem}
If $n$ and $p$ are non-negative integers, then
\begin{equation}
\begin{split}
&\sum_{k = 0}^n {\frac{1}{{2^{2k} }}\binom{{2(k + p)}}{{k + p}}\binom{{k + p}}{k}H_{n - k} \left( m \right)\left( {O_{k + p}  - O_p } \right)}\\
&\qquad  = \frac{1}{2}\sum_{k = 0}^n {\frac{1}{{2^{2k} }}\binom{{2(k + p)}}{{k + p}}\binom{{k + p}}{k}H_{n - k} \left( {m + 1} \right)} .
\end{split}
\end{equation}
\end{theorem}
In particular, setting $m=0$ and using~\eqref{eq.suzj3to} gives the following generalization of~\eqref{eq.tb6ik5l}:
\begin{equation}
\begin{split}
&\sum_{k = 0}^n {\frac{1}{{2^{2k} }}\binom{{2(k + p)}}{{k + p}}\binom{{k + p}}{k}H_{n - k}}\\
&\qquad=\frac{1}{{2^{2n} }}\,\frac{{p + 1}}{{2p + 1}}\binom{{2\left( {n + p + 1} \right)}}{{n + p + 1}}\binom{{n + p + 1}}{n}\left( {O_{n + p + 1}  - O_{p + 1} } \right).
\end{split}
\end{equation}

\begin{proof}
Write $p-1/2$ for $p$ in~\eqref{eq.dk6gvdp} and use~\eqref{eq.tb96ki2} and~\eqref{eq.j80dqx2}.
\end{proof}

\begin{lemma}\label{lem.qrsgpmt}
If $(a_n)_{n\geq 0}$ is a sequence, then
\begin{equation*}
\sum_{k = 1}^n k (a_k - a_{k - 1}) = n a_n - \sum_{k = 1}^n a_{k - 1}.
\end{equation*}
\end{lemma}
\begin{theorem}
If $n\in\mathbb N_0$ and $p\in\mathbb C\setminus\mathbb Z^{-}$ then 
\begin{equation}\label{eq.xld8bhi}
\sum_{k = 1}^n {kH_{k,p} }  = nH_{n,p + 1}  - H_{n - 1,p + 2}.
\end{equation}
\end{theorem}
\begin{proof}
Identity~\eqref{eq.sw1cyzy} gives the following recurrence relation:
\begin{equation}\label{eq.unfun08}
H_{k,p + 1}  - H_{k - 1,p + 1}  = H_{k,p}.
\end{equation}
Use $a_k = H_{k,p + 1}$ in Lemma~\ref{lem.qrsgpmt}, keeping~\eqref{eq.unfun08} in mind.
\end{proof}

\begin{theorem}
If $n$ and $p$ are non-negative integers, then
\begin{equation}
\begin{split}
&\sum_{k = 1}^n {\frac{k}{{2^{2k} }}\binom{{2(k + p)}}{{(k + p)}}\binom{{k + p}}{k}\left( {O_{k + p}  - O_p } \right)} \\
&\qquad = \frac{n}{{2^{2n} }}\binom{{2(p + 1)}}{{p + 1}}^{ - 1} \binom{{2p}}{p}\binom{{2(n + p + 1)}}{{n + p + 1}}\binom{{n + p + 1}}{n}\left( {O_{n + p + 1}  - O_{p + 1} } \right)\\
&\qquad\qquad - \frac{1}{{2^{2n - 2} }}\binom{{2(p + 2)}}{{p + 2}}^{ - 1} \binom{{2p}}{p}\binom{{2(n + p + 1)}}{{n + p + 1}}\binom{{n + p + 1}}{{n - 1}}\left( {O_{n + p + 1}  - O_{p + 2} } \right).
\end{split}
\end{equation}
\end{theorem}
In particular,
\begin{equation}
\begin{split}
\sum_{k = 1}^n {\frac{k}{{2^{2k} }}\binom{{2k}}{k}O_k }  &= \frac{{n(n + 1)}}{{2^{2n + 1} }}\binom{{2(n + 1)}}{{n + 1}}\left( {O_{n + 1}  - 1} \right)\\
&\qquad - \frac{{n(n + 1)}}{{2^{2n} 3}}\binom{{2(n + 1)}}{{n + 1}}\left( {O_{n + 1}  - \frac{4}{3}} \right).
\end{split}
\end{equation}
\begin{proof}
Write $p+1/2$ for $p$ in~\eqref{eq.xld8bhi} and use~\eqref{eq.m2jjbl5}.
\end{proof}

\begin{theorem}
If $n$ and $p$ are non-negative integers, then
\begin{equation}\label{eq.yycg1tg}
\begin{split}
&\sum_{k = 1}^n {\binom{{2k}}{k}^{ - 1} \binom{{2(k + p)}}{{k + p}}\binom{{k + p}}{k}\left( {O_{k + p}  - O_k } \right)} \\
&\qquad = \frac{1}{4}\binom{{2n}}{n}^{ - 1} \binom{{2\left( {n + p + 1} \right)}}{{n + p + 1}}\binom{n + p + 1}n\left( {O_{n + p + 1}  - O_n } \right) - \frac{1}{4}\binom{{2\left( {p + 1} \right)}}{{p + 1}}O_{p + 1} .
\end{split}
\end{equation}
\end{theorem}
\begin{proof}
Rearrange~\eqref{eq.unfun08}, interchange $k$ and $p$ and write $k + 1/2$ for $k$ to obtain
\begin{equation*}
H_{p,k + 1/2}  = H_{p + 1,k + 1/2}  - H_{p + 1,k - 1/2}
\end{equation*}
which telescopes to give
\begin{equation*}
\sum_{k = 1}^n {H_{p,k + 1/2} }  = H_{p + 1,n + 1/2}  - H_{p + 1,1/2},
\end{equation*}
from which~\eqref{eq.yycg1tg} follows, in view of~\eqref{eq.m2jjbl5}.
\end{proof}

\end{document}